\documentclass[a4paper, 12pt,oneside,reqno]{amsart}
\usepackage[a4paper]{geometry}
\usepackage{amssymb,amsmath,amsthm}
\usepackage[T2A]{fontenc}
\usepackage[breaklinks=true]{hyperref}

\usepackage{xcolor}

\usepackage{bbm}
\usepackage{graphicx}

\DeclareSymbolFont{symbolsC}{U}{pxsyc}{m}{n}
\SetSymbolFont{symbolsC}{bold}{U}{pxsyc}{bx}{n}
\DeclareFontSubstitution{U}{pxsyc}{m}{n}
\DeclareMathSymbol{\medcirc}{\mathbin}{symbolsC}{7}
\DeclareMathSymbol{\medbullet}{\mathbin}{symbolsC}{8}

\def\dim{\mathrm{dim}}

\let\<\langle
\let\>\rangle

\newtheorem{definition}{Definition}
\newtheorem{lemma}{Lemma}

\newtheorem{theorem}{Theorem}
\newtheorem{corollary}{Corollary}
\newtheorem{remark}{Remark}
\newtheorem{example}{Example}

\title{Free Novikov-Zinbiel algebra}

\author{A. Dauletiyarova, F. Mashurov, and B. Sartayev}

\address{Institute of Mathematics and Mathematical Modeling, Almaty, Kazakhstan}

\address{Narxoz University, Almaty, Kazakhstan}

\email{d\_aigera95@mail.ru}

\address{SICM, Southern University of Science and Technology, Shenzhen, 518055, China}

\address{SDU University, Kaskelen, Kazakhstan}

\email{f.mashurov@gmail.com}

\address{Narxoz University, Almaty, Kazakhstan}

\address{SDU University, Almaty, Kazakhstan}

\email{baurjai@gmail.com}


\thanks{${}^{*}$Corresponding author: Bauyrzhan Sartayev   (baurjai@gmail.com)}

\begin{document}

\begin{abstract}
Let $A$ be a commutative-associative algebra with an invertible derivation $D$, and put $R=D^{-1}$. We study the operations
\begin{equation*}
x\succ y=R(x)y,\qquad x\prec y=xD(y),
\end{equation*}
which define a Novikov-Zinbiel algebra. Using a simple graded model, we represent multilinear $\prec,\succ$-monomials by rational functions and construct an explicit basis for the resulting space. This yields a description of the multilinear components of the free Novikov-Zinbiel algebra. As a consequence, we prove that every multilinear identity satisfied by this construction follows from the defining identities of Novikov-Zinbiel algebras.
\end{abstract}

\maketitle

\section{Introduction}

Novikov algebras appeared in the study of Hamiltonian operators and Poisson brackets of hydrodynamic type \cite{BalNovikov1985,GelDorfman1979}. Their structure was further studied in \cite{BurdeGraaf2013,KaygorodovMashurovNamZhang2024,Osborn1992,Xu2001,Zelmanov1987}.  Recent developments also include pre-Novikov bialgebras, quasi-Frobenius structures, and their affinization constructions; see \cite{LiHong2025,LiHong2026}. One of the standard constructions of a Novikov algebra starts with a commutative-associative algebra equipped with a derivation. Up to the choice of the opposite convention, the corresponding product has the form
\begin{equation*}
x\prec y=xD(y).
\end{equation*}
The differential realization is also closely related to the structure of free Novikov algebras. Bases, identities, and representations of free Novikov algebras were studied, in particular, in \cite{DzhIsmailov2014,DzhLofwall2002}.

Zinbiel algebras were introduced by Loday under the name of dual Leibniz algebras \cite{Loday1995}. They are Koszul dual to Leibniz algebras \cite{GinzKapranov1994,LodayVallette2012}. Various structural properties and identities of Zinbiel algebras were studied in \cite{DzhTulenbaev2005,NaurUmirbaev2010}. Zinbiel-type constructions are also naturally related to Rota--Baxter operators \cite{DzhZinbiel2007,GuoKeigher2008}.

The Novikov and Zinbiel constructions become closely related when the derivation is invertible. Let $A$ be a commutative-associative algebra with an invertible derivation $D$, and put $R=D^{-1}$. Then
\begin{equation*}
R(x)R(y)=R\bigl(xR(y)+R(x)y\bigr),
\end{equation*}
so $R$ satisfies the Rota--Baxter identity of weight zero. Thus the derivation and its inverse naturally define two operations on the same algebra:
\begin{equation*}
x\succ y=R(x)y,\qquad x\prec y=xD(y).
\end{equation*}
Related constructions involving differential and Rota--Baxter operators and Novikov structures have been studied in \cite{GaoGuoHanZhang2025}. Constructions combining derivation and integration-type operators in nonassociative algebras were also considered in \cite{KaygorodovUzakbaev2026}. Related constructions involving derivations and Rota--Baxter operators, as well as operadic properties of Novikov-type varieties, have also been studied in \cite{KolesnikovMashurovSartayev2024,KolesnikovSartayev2026}.

A direct computation shows that $\succ$ satisfies the Zinbiel identity, while $\prec$ satisfies the two Novikov identities. Moreover, the two operations satisfy three compatibility identities. We call an algebra satisfying these six cubic identities a Novikov-Zinbiel algebra. 



The main problem of this paper is to determine whether  these six identities completely describe the identities of the above construction. Our main result gives a positive answer in the multilinear case: every multilinear identity satisfied by
\begin{equation*}
x\succ y=R(x)y,
\qquad
x\prec y=xD(y)
\end{equation*}
follows from the Zinbiel identity, the two Novikov identities, and the three compatibility identities. In order to prove this result, we also obtain an explicit description of the multilinear components of the free Novikov-Zinbiel algebra.

Our main tool is a rational function model. We consider
\begin{equation*}
A=t\Bbbk[t],
\qquad
D=t\frac{d}{dt}.
\end{equation*}
For positive integers $p$ and $q$, we have
\begin{equation*}
t^p\prec t^q=q\,t^{p+q},
\qquad
t^p\succ t^q=\frac{1}{p}t^{p+q}.
\end{equation*}
Hence, after the substitution $x_i=t^{p_i}$, every multilinear $\prec,\succ$-monomial is represented by a rational function in $p_1,\ldots,p_n$. More precisely, every occurrence of $\prec$ contributes the weight of a subtree to the numerator, while every occurrence of $\succ$ contributes the inverse of a subtree weight. Thus the problem of polynomial identities is translated into a problem about a concrete space of rational functions.

The first main result is an explicit basis of this rational function space. The basis is described by strict chains of subsets of $[n]$ together with polynomial monomials attached to the corresponding blocks. The proof uses a signed rooted-tree representation of the rational functions, cancellation of denominator factors, partial fraction decompositions, and residues. As an application, we obtain a recurrence for the dimensions of the multilinear components.

We then return to the free Novikov-Zinbiel algebra. For every element of the rational basis we construct a canonical lift to the free algebra. We prove that these lifts form a basis and that the rational function model is faithful in every multilinear degree. This gives the completeness of the six defining cubic identities.

The paper is organized as follows. We first derive the defining identities and give several examples. We then introduce the rational function model, construct its basis, and obtain the dimension recurrence. In the final section, we construct normal forms in the free Novikov-Zinbiel algebra and prove the completeness of the defining identities.

\section{Algebras with an invertible derivation}

Let $\Bbbk$ be a field of characteristic zero and let $A$ be a commutative associative $\Bbbk$-algebra. We assume that $A$ is endowed with an invertible derivation $D$, and denote its inverse by $R=D^{-1}$. Thus $DR=RD=\mathrm{id}$. We do not assume that $A$ is unital. Indeed, if $A$ has a unit, then $D(1)=0$, and hence $D$ cannot be invertible unless $A=0$.

We define two binary operations on $A$ by
\[
x\succ y=R(x)y,\qquad x\prec y=xD(y).
\]

Our aim is to describe the polynomial identities satisfied by these operations. We start with multilinear identities of degree three. First, observe that
\[
R(x)R(y)=R\bigl(xR(y)+R(x)y\bigr).
\]
Indeed, by the Leibniz rule, $D(R(x)R(y))=xR(y)+R(x)y$, and the equality follows after applying $R=D^{-1}$. This gives
\[
x\succ(y\succ z)=(x\succ y)\succ z+(y\succ x)\succ z.
\]
As expected, the operation $\succ$ satisfies the Zinbiel identity.

Now consider the operation $\prec$. Since $(x\prec y)\prec z=xD(y)D(z)$, commutativity of the multiplication in $A$ gives
\[
(x\prec y)\prec z=(x\prec z)\prec y.
\]
Further, $x\prec(y\prec z)=xD(yD(z))=xD(y)D(z)+xyD^2(z)$. Hence
\[
(x\prec y)\prec z-x\prec(y\prec z)=-xyD^2(z).
\]
The right-hand side is symmetric in $x$ and $y$, and therefore
\[
(x\prec y)\prec z-x\prec(y\prec z)=(y\prec x)\prec z-y\prec(x\prec z).
\]
Thus $\prec$ is a Novikov product.

We next consider the mixed identities. From the definitions,
$(x\succ y)\prec z=R(x)yD(z)$ and
$x\succ(y\prec z)=R(x)yD(z)$. Therefore
\[
(x\succ y)\prec z=x\succ(y\prec z).
\]
For the next relation, we compute
$x\prec(y\succ z)=xD(R(y)z)$. Since $D(R(y))=y$, we obtain
\[
x\prec(y\succ z)=xyz+xR(y)D(z)=xyz+(y\succ x)\prec z.
\]
The ordinary product $xyz$ can also be expressed in terms of
$\prec$ and $\succ$. Indeed,
\[
(x\prec y)\succ z+(y\prec x)\succ z=R\bigl(xD(y)+yD(x)\bigr)z=R(D(xy))z=xyz.
\]
Hence
\[
x\prec(y\succ z)=(x\prec y)\succ z+(y\prec x)\succ z+(y\succ x)\prec z.
\]
Since the monomial $xyz$ can be written by operations $\succ$ and $\prec$, the expression of $xyz$ is symmetric in $x,y,z$. Thus
\[
(x\prec y)\succ z+(y\prec x)\succ z=(x\prec z)\succ y+(z\prec x)\succ y.
\]

\begin{definition}
An algebra $A$ with two binary operations $\prec$ and $\succ$ is called a Novikov-Zinbiel algebra if $A$ is a Novikov algebra with respect to $\prec$, a Zinbiel algebra with respect to $\succ$, and the following compatibility identities hold:
\begin{equation}\label{eq:NZ1}
(x\succ y)\prec z=x\succ(y\prec z),
\end{equation}
\begin{equation}\label{eq:NZ2}
x\prec(y\succ z)=(x\prec y)\succ z+(y\prec x)\succ z+(y\succ x)\prec z,
\end{equation}
and
\begin{equation}\label{eq:NZ3}
(x\prec y)\succ z+(y\prec x)\succ z=(x\prec z)\succ y+(z\prec x)\succ y.
\end{equation}
\end{definition}

\begin{example}
Let $A=\bigoplus_{n\geq 1}A_n$ be a positively graded commutative associative algebra such that $A_iA_j\subseteq A_{i+j}$. Define $D$ on homogeneous elements by $D(a)=na$ for $a\in A_n$. Then $D$ is an invertible derivation. Its inverse is given by $R(a)=\frac{1}{n}a$ for $a\in A_n$.

Hence, for homogeneous elements $a\in A_p$ and $b\in A_q$, the operations
\begin{equation}\label{eq:gradedNZ}
a\prec b=q\,ab,\qquad a\succ b=\frac{1}{p}ab
\end{equation}
define a Novikov-Zinbiel algebra structure on $A$.
\end{example}

The following examples are particular cases of the previous construction.

\begin{example}\label{ex2}
Let $A=t\Bbbk[t]$ with the standard grading $\deg t^n=n$. Then
$A=\bigoplus_{n\geq 1}\Bbbk t^n$. The derivation from the previous example is
\[
D=t\frac{d}{dt},
\]
since $D(t^n)=nt^n$. Therefore $R(t^n)=\frac{1}{n}t^n$, and
\begin{equation}\label{eq:onevarNZ}
t^p\prec t^q=q\,t^{p+q},\qquad t^p\succ t^q=\frac{1}{p}t^{p+q}.
\end{equation}
Thus this example is the one-dimensional graded case of the general construction above, where each homogeneous component is one-dimensional.
\end{example}

\begin{example}
For $N\geq 1$, let
\[
A_N=t\Bbbk[t]/(t^{N+1}).
\]
It inherits the positive grading $A_N=\bigoplus_{n=1}^N \Bbbk t^n$. Hence it is again a particular case of the first example. Writing $e_i=t^i$, we have
\[
D(e_i)=i e_i,\qquad R(e_i)=\frac{1}{i}e_i.
\]
The corresponding Novikov-Zinbiel products are
\begin{equation*}
e_i\prec e_j=
\begin{cases}
j e_{i+j}, & i+j\leq N,\\
0, & i+j>N,
\end{cases}
\end{equation*}
and
\begin{equation*}
e_i\succ e_j=
\begin{cases}
\frac{1}{i}e_{i+j}, & i+j\leq N,\\
0, & i+j>N.
\end{cases}
\end{equation*}
Thus $A_N$ is a finite-dimensional truncation of the previous example. For example, if $N=3$, the nonzero products are 
\[
e_1\prec e_1=e_2,\qquad e_1\prec e_2=2e_3,\qquad e_2\prec e_1=e_3,\]
\[ e_1\succ e_1=e_2,\qquad e_1\succ e_2=e_3,\qquad e_2\succ e_1=\frac{1}{2}e_3.
\]
\end{example}

\begin{example}
Let
\[
A=(x_1,\ldots,x_m)\Bbbk[x_1,\ldots,x_m],
\]
and let $w_1,\ldots,w_m$ be positive integers. Give $A$ the weighted grading determined by $\deg x_i=w_i$. Then a monomial $x^\alpha=x_1^{\alpha_1}\cdots x_m^{\alpha_m}$ has degree
\[
|\alpha|_w=\sum_{i=1}^m w_i\alpha_i.
\]
The derivation corresponding to this grading is the weighted Euler derivation

\[
D=\sum_{i=1}^m w_i x_i\frac{\partial}{\partial x_i}\;\;\textrm{and}\;\;D(x^\alpha)=|\alpha|_w x^\alpha.
\]
Since all weights $w_i$ are positive and $A$ contains no nonzero elements of weighted degree zero, the operator $D$ is invertible on $A$. Thus this is again a particular case of the first example. Since all nonzero homogeneous components have positive degree, $D$ is invertible on $A$, and
\begin{equation*}
x^\alpha\prec x^\beta=|\beta|_w x^{\alpha+\beta},\qquad
x^\alpha\succ x^\beta=\frac{1}{|\alpha|_w}x^{\alpha+\beta}.
\end{equation*}
For $m=1$ and $w_1=1$, this reduces to the algebra $t\Bbbk[t]$ from Example \ref{ex2}.
\end{example}

\begin{remark}
One may also consider the operations
\[
x\prec_1 y=xR(y),\qquad x\prec_2 y=xD(y).
\]
However, this does not lead to a new operad. Indeed, by commutativity of the underlying associative product,
\[
x\succ y=R(x)y=yR(x)=y\prec_1 x,
\qquad
x\prec y=x\prec_2 y.
\]
Therefore the operad with operations $\prec_1$ and $\prec_2$
is isomorphic to the Novikov-Zinbiel operad. Consequently, there is no need to study this variant separately.
\end{remark}

\section{The rational function model}

To study multilinear identities we use the algebra $A=t\Bbbk[t]$ with the derivation $D=t\frac{d}{dt}$ from Example \ref{ex2}. For $n\geq 1$, we have $D(t^n)=nt^n$ and $R(t^n)=\frac{1}{n}t^n$. Therefore,
\begin{equation*}
t^p\prec t^q=q\,t^{p+q},\qquad t^p\succ t^q=\frac{1}{p}t^{p+q}.
\end{equation*}

Let $T(x_1,\ldots,x_n)$ be a nonassociative monomial in the operations $\prec$ and $\succ$. Substituting $x_i=t^{p_i}$, we obtain
\begin{equation*}
T(t^{p_1},\ldots,t^{p_n})=c_T(p_1,\ldots,p_n)t^{p_1+\cdots+p_n},
\end{equation*}
where $c_T(p_1,\ldots,p_n)$ is a rational function.

For a submonomial $U$, denote by $p_U$ the sum of the variables $p_i$ corresponding to the leaves of $U$. If $T=U\prec V$ or $T=U\succ V$, then
\begin{equation*}
c_{U\prec V}=p_Vc_Uc_V,\qquad c_{U\succ V}=\frac{c_Uc_V}{p_U}.
\end{equation*}
Thus the rational function corresponding to a monomial can be read directly from its binary tree: every vertex labelled by $\prec$ contributes the weight of its right subtree, while every vertex labelled by $\succ$ contributes the inverse of the weight of its left subtree.

In degree three, putting $x=t^p$, $y=t^q$ and $z=t^r$, we obtain

\begin{equation*}
\begin{array}{c|cccc}
T
&(x\succ y)\succ z
&x\succ(y\succ z)
&(x\succ y)\prec z
&x\succ(y\prec z)
\\ \hline
c_T
&\frac1{p(p+q)}
&\frac1{pq}
&\frac rp
&\frac rp
\\[4mm]
T
&(x\prec y)\succ z
&x\prec(y\succ z)
&(x\prec y)\prec z
&x\prec(y\prec z)
\\ \hline
c_T
&\frac q{p+q}
&\frac{q+r}{q}
&qr
&r(q+r)
\end{array}
\end{equation*}
Together with all permutations of $p,q,r$, these rational functions represent all multilinear monomials of degree three. Hence every multilinear identity gives a linear relation between the corresponding rational functions.

More generally, let $\mathcal F(n)$ denote the multilinear component of the free algebra with two binary operations $\prec$ and $\succ$. The monomial model defines a linear map
\begin{equation*}
\rho_n:\mathcal F(n)\longrightarrow \Bbbk(p_1,\ldots,p_n),\qquad \rho_n(T)=c_T.
\end{equation*}
The kernel of $\rho_n$ is the space of multilinear identities satisfied by the operations $\prec$ and $\succ$ in the monomial model, while its image is the space spanned by the rational functions $c_T$.

For $n=3$, we have $\dim(\mathcal F(3))=48$.
The mixed component in degree three consists of monomials containing one occurrence of $\prec$ and one occurrence of $\succ$. Up to permutations of the variables, the corresponding rational functions are of the following three types:
\begin{equation*}
\frac{r}{p},\qquad \frac{q}{p+q},\qquad \frac{q+r}{q}.
\end{equation*}
So, the mixed component is spanned by the following ten functions:
\begin{equation}\label{eq:mixed-basis-degree-three}
1,\qquad
\frac pq,\frac pr,\frac qp,\frac qr,\frac rp,\frac rq,\qquad
\frac{p}{p+q},\frac{p}{p+r},\frac{q}{q+r}.
\end{equation}
Indeed, the functions in \eqref{eq:mixed-basis-degree-three} satisfy
\begin{equation*}
\frac{q}{p+q}=1-\frac{p}{p+q},\;\;\;
\frac{r}{p+r}=1-\frac{p}{p+r},\;\;\;
\frac{r}{q+r}=1-\frac{q}{q+r}.
\end{equation*}

These ten functions from \eqref{eq:mixed-basis-degree-three} are linearly independent. Indeed, the first seven functions are Laurent monomials and are linearly independent. The last three functions have poles on the three different hyperplanes
$p+q=0$, $p+r=0$, and $q+r=0$, respectively, while the first seven functions are regular on these hyperplanes.

For example, suppose that
\[
a\frac{p}{p+q}
+b\frac{p}{p+r}
+c\frac{q}{q+r}
+f(p,q,r)=0,
\]
where $f$ is a linear combination of the first seven functions in \eqref{eq:mixed-basis-degree-three}. Multiplying by $p+q$ and setting $q=-p$, we obtain $ap=0$, and hence $a=0$. In the same way, using the hyperplanes $p+r=0$ and $q+r=0$, we get $b=c=0$. It follows that $f=0$, and therefore all coefficients are zero. Consequently,
\begin{equation*}
\dim\mathcal{R}_{3,1}=10.
\end{equation*}

The rational functions corresponding to these $48$ monomials span a space of dimension $22$. Hence
\begin{equation*}
\dim(\ker\rho_3)=48-22=26.
\end{equation*}

On the other hand, the $S_3$-orbits of the Zinbiel identity, the two Novikov identities, and the three compatibility identities \eqref{eq:NZ1}--\eqref{eq:NZ3} span a $26$-dimensional subspace of $\mathcal F(3)$. Therefore, these identities generate all multilinear identities of degree three.

In particular, if $\mathcal{NZ}$ denotes the operad of Novikov-Zinbiel algebras, then
\begin{equation*}
\dim\mathcal{NZ}(3)=22.
\end{equation*}
Thus, the defining identities of a Novikov-Zinbiel algebra form a complete system of multilinear identities of degree three for the construction coming from an invertible derivation.

It is useful to note that not every homogeneous rational function of degree zero belongs to the mixed component. For example,
\begin{equation*}
\frac{r}{p+q}\notin \mathcal{R}_{3,1}.
\end{equation*}
Indeed, consider a linear relation over $\Bbbk$ of the form
\begin{equation*}
\frac{r}{p+q}=a\frac{p}{p+q}+g(p,q,r),
\end{equation*}
where $g$ is a linear combination of the remaining functions in
\eqref{eq:mixed-basis-degree-three}. None of the denominators occurring
in $g$ is divisible by $p+q$. Hence $g$ is regular at a generic point
of the hyperplane $p+q=0$.
Multiplying both sides by $p+q$ and restricting to
$p+q=0$, we obtain
\begin{equation}\label{eq:r-ap}
r=ap.
\end{equation}
This is impossible, since $a\in\Bbbk$ and $p$ and $r$ are independent
variables on the hyperplane $p+q=0$. Therefore
$\frac{r}{p+q}$ does not belong to $\mathcal{R}_{3,1}$.

\section{A basis of the rational function model}

For every nonempty subset $I\subseteq[n]=\{1,\ldots,n\}$, put $p_I=\sum_{i\in I}p_i$. Let
\begin{equation*}
\mathcal R(n)=\operatorname{Im}\rho_n
\subseteq\Bbbk(p_1,\ldots,p_n).
\end{equation*}
We decompose $\mathcal R(n)$ according to the number of occurrences of
$\prec$:
\begin{equation*}
\mathcal R(n)=\bigoplus_{k=0}^{n-1}\mathcal R_{n,k},
\end{equation*}
where $\mathcal R_{n,k}$ is spanned by the rational functions
corresponding to multilinear monomials containing exactly $k$
occurrences of $\prec$. Since every such monomial contains $n-1-k$
occurrences of $\succ$, every element of $\mathcal R_{n,k}$ is
homogeneous of degree
\begin{equation*}
k-(n-1-k)=2k-n+1.
\end{equation*}

We now define a set $\mathcal B_{n,k}$ of rational functions. Choose a
strict chain
\begin{equation*}
\varnothing=I_0\subsetneq I_1\subsetneq\cdots
\subsetneq I_j\subsetneq I_{j+1}=[n]
\end{equation*}
and put
\begin{equation*}
B_s=I_s\setminus I_{s-1},
\qquad
1\leq s\leq j+1.
\end{equation*}
Choose integers $d_1,\ldots,d_{j+1}$ satisfying
\begin{equation}\label{eq:basis-degree-conditions}
0\leq d_s\leq |B_s|-1,
\qquad
d_s\equiv |B_s|-1\pmod 2,
\end{equation}
and
\begin{equation}\label{eq:basis-total-degree}
d_1+\cdots+d_{j+1}=2k+j-n+1.
\end{equation}
For each $1\leq s\leq j$, let $M_s$ be a monomial of total degree
$d_s$ in the variables
\begin{equation*}
\{p_a\mid a\in B_s\setminus\{\max B_s\}\},
\end{equation*}
and let $M_{j+1}$ be a monomial of total degree $d_{j+1}$ in the
variables
\begin{equation*}
\{p_a\mid a\in B_{j+1}\}.
\end{equation*}
To these data we associate the rational function
\begin{equation}\label{eq:basis-element-Rnk}
\frac{M_1M_2\cdots M_{j+1}}
{p_{I_1}p_{I_2}\cdots p_{I_j}}.
\end{equation}
Let $\mathcal B_{n,k}$ be the set of all functions obtained in this
way.

\begin{theorem}\label{thm:basis}
For fixed $n$ and $k$, the set $\mathcal B_{n,k}$ is a basis of $\mathcal R_{n,k}$. Consequently, $\bigcup_{k=0}^{n-1}\mathcal B_{n,k}$ is a basis of $\mathcal R(n)$.
\end{theorem}

The description of the basis has a simple interpretation. The strict chain $I_1\subsetneq\cdots\subsetneq I_j$ records the denominator factors $p_{I_1}\cdots p_{I_j}$, while the blocks $B_s=I_s\setminus I_{s-1}$ determine the polynomial contributions to the numerator.

If $|B_s|=r$, then the maximal polynomial degree contributed by this block is $r-1$, and each cancellation of a $\succ$ with a $\prec$ lowers the degree by two. Hence
\[
0\leq d_s\leq |B_s|-1,
\qquad
d_s\equiv |B_s|-1\pmod 2.
\]
Since \eqref{eq:basis-element-Rnk} has total degree $d_1+\cdots+d_{j+1}-j$,
the condition that it belong to $\mathcal R_{n,k}$ is precisely
\[
d_1+\cdots+d_{j+1}=2k+j-n+1.
\]
For a nonfinal block $B_s$, the relation
\[
p_{B_s}=\sum_{a\in B_s}p_a=0
\]
allows one variable to be eliminated; we choose $p_{\max B_s}$. Thus $M_s$ involves the variables indexed by
$B_s\setminus{\max B_s}$. For the final block no such relation occurs, so $M_{j+1}$ may involve all variables indexed by $B_{j+1}$.

\begin{example}
Let $n=3$ and $k=1$. Then every element of $\mathcal R_{3,1}$ is homogeneous of degree $2k-n+1=0$. We determine the elements of $\mathcal B_{3,1}$ directly from the chain description.

First, let $j=0$. Then the chain is simply
\begin{equation*}
\varnothing=I_0\subsetneq I_1=[3],\;\textrm{so}\;B_1=\{1,2,3\}.
\end{equation*}
The total degree condition gives $d_1=2k-n+1=0$. Thus $M_1=1$, and the corresponding basis element is $1$.
This function is indeed obtained from $\prec$ and $\succ$. For
example,
\begin{equation*}
\rho_3\bigl((x\prec y)\succ z+(y\prec x)\succ z\bigr)=1.
\end{equation*}

Now let $j=1$. We have a chain
\begin{equation*}
\varnothing\subsetneq I_1\subsetneq I_2=[3],\;\textrm{with}\; B_1=I_1,\; B_2=[3]\setminus I_1.
\end{equation*}
The total degree condition becomes
$d_1+d_2=1$.

Suppose first that $|B_1|=1$ and $|B_2|=2$.
The conditions \eqref{eq:basis-degree-conditions} force
$d_1=0$ and $d_2=1$. Thus $M_1=1$, while $M_2$ is any variable belonging to $B_2$.

For $I_1=B_1=\{1\}$ and $B_2=\{2,3\}$, we obtain
\begin{equation*}
\frac{q}{p},
\qquad
\frac{r}{p}.
\end{equation*}
Similarly, the choices $I_1=\{2\}$ and $I_1=\{3\}$ give
\begin{equation*}
\frac{p}{q},
\qquad
\frac{r}{q},
\qquad
\frac{p}{r},
\qquad
\frac{q}{r}.
\end{equation*}

Suppose now that $|B_1|=2$ and $|B_2|=1$. Then the degree conditions force $d_1=1$ and $d_2=0$. For
$I_1=B_1=\{1,2\}$ and $B_2=\{3\}$, the nonfinal block $B_1$ uses only the variables
\begin{equation*}
\{p_a\mid a\in B_1\setminus\{\max B_1\}\}=\{p\}.
\end{equation*}
Hence $M_1=p$, and we obtain
\begin{equation*}
\frac{p}{p+q}.
\end{equation*}
The choices $I_1=\{1,3\}$ and $I_1=\{2,3\}$ similarly give
\begin{equation*}
\frac{p}{p+r},
\qquad
\frac{q}{q+r}.
\end{equation*}

There are no basis elements with $j=2$. Indeed, in this case all three
blocks are singletons, so the degree conditions force
\begin{equation*}
d_1=d_2=d_3=0,
\end{equation*}
while \eqref{eq:basis-total-degree} would require
\begin{equation*}
d_1+d_2+d_3=2.
\end{equation*}
Thus
\begin{equation*}
\mathcal B_{3,1}
=
\left\{
1,\,
\frac{q}{p},\frac{r}{p},
\frac{p}{q},\frac{r}{q},
\frac{p}{r},\frac{q}{r},\,
\frac{p}{p+q},
\frac{p}{p+r},
\frac{q}{q+r}
\right\}.
\end{equation*}
Therefore $\dim\mathcal R_{3,1}=10$. For $n=3$, the three homogeneous components have dimensions
\begin{equation*}
\dim\mathcal R_{3,0}=6,
\qquad
\dim\mathcal R_{3,1}=10,
\qquad
\dim\mathcal R_{3,2}=6,
\end{equation*}
and hence
\begin{equation*}
\dim\mathcal R(3)=22.
\end{equation*}
\end{example}

The same construction gives
\begin{equation*}
\dim\mathcal R(n)
=
1,\ 4,\ 22,\ 152,\ 1291,\ 13156,
\end{equation*}
for $n=1,\ldots,6$.

\section{The proof of Theorem \ref{thm:basis}}

Let $\mathcal B_{n,k}$ be the set of rational functions defined in the
previous section. We first prove several lemmas.

\begin{lemma}\label{lem:pure-prec-polynomials}
Let $B$ be a set of $r$ variables. The coefficients of multilinear
polynomials involving only $\prec$ span the space of all homogeneous
polynomials of degree $r-1$ in the variables $p_a$, $a\in B$.
\end{lemma}

\begin{proof}
By the standard differential realization of the free Novikov algebra \cite{DzhLofwall2002},
its multilinear component is identified with the span of the
differential monomials
\begin{equation}\label{eq:pure-prec-differential}
\prod_{a\in B}D^{\alpha_a}(x_a),
\qquad
\sum_{a\in B}\alpha_a=r-1.
\end{equation}
Under the substitution $x_a=t^{p_a}$, we have
$D^{\alpha_a}(t^{p_a})=p_a^{\alpha_a}t^{p_a}$.
Hence \eqref{eq:pure-prec-differential} has coefficient $\prod_{a\in B}p_a^{\alpha_a}$. These are precisely all monomials of degree $r-1$. Therefore they span the whole space of homogeneous polynomials of degree $r-1$.
\end{proof}

The following elementary cancellation will be used repeatedly.

\begin{lemma}\label{lem:hidden-pair}
Let $U,V,W$ involve pairwise disjoint sets of variables. Put
\begin{equation*}
\mu(U,V,W)
=
(U\prec V)\succ W+(V\prec U)\succ W.
\end{equation*}
Then
\begin{equation}\label{eq:mu-coefficient}
c_{\mu(U,V,W)}=c_Uc_Vc_W.
\end{equation}
\end{lemma}

\begin{proof}
We have
\begin{equation*}
c_{(U\prec V)\succ W}=\frac{p_V}{p_U+p_V}c_Uc_Vc_W\;\;\textrm{and}\;\;c_{(V\prec U)\succ W}=\frac{p_U}{p_U+p_V}c_Uc_Vc_W.
\end{equation*}
Adding the two expressions gives \eqref{eq:mu-coefficient}.
\end{proof}

\begin{lemma}\label{lem:block-realization}
Let $B$ be a set of $r$ variables and let $d$ satisfy
\begin{equation*}
0\leq d\leq r-1,\qquad d\equiv r-1\pmod 2.
\end{equation*}
Then every monomial $M$ of degree $d$ in the variables $p_a$, $a\in B$, is the coefficient of a multilinear $\prec,\succ$-polynomial on the variables indexed by $B$. This polynomial contains exactly
\begin{equation*}
\frac{r-1-d}{2}
\end{equation*}
occurrences of $\succ$.
\end{lemma}

\begin{proof}
Put $h=\frac{r-1-d}{2}$. By the assumptions on $d$, the integer $h$ is nonnegative. Since $M$ has degree $d$, it involves at most $d$ different variables.
Choose a subset
\begin{equation*}
C\subseteq B,
\qquad
|C|=d+1,
\end{equation*}
containing all variables occurring in $M$.

By Lemma \ref{lem:pure-prec-polynomials}, there exists a multilinear polynomial $P_0$ on the variables indexed by $C$, involving only $\prec$, such that $c_{P_0}=M$. It contains exactly $d$ occurrences of $\prec$. Since
\begin{equation*}
|B\setminus C|=r-d-1=2h,
\end{equation*}
write
\begin{equation*}
B\setminus C=\{a_1,b_1,\ldots,a_h,b_h\}.
\end{equation*}
Define successively
\begin{equation*}
P_s=\mu(P_{s-1},x_{a_s},x_{b_s}),\qquad1\leq s\leq h.
\end{equation*}
By Lemma \ref{lem:hidden-pair}, $c_{P_s}=c_{P_{s-1}}$. Hence
\begin{equation*}
c_{P_h}=M.
\end{equation*}
Each step adds one occurrence of $\prec$ and one occurrence of
$\succ$. Thus $P_h$ contains $h$ occurrences of $\succ$ and
$d+h=r-1-h$ occurrences of $\prec$.
\end{proof}

We can now realize every element of the proposed basis.

\begin{lemma}\label{lem:basis-in-R}
For every $n$ and $k$,
\begin{equation*}
\mathcal B_{n,k}\subseteq\mathcal R_{n,k}.
\end{equation*}
\end{lemma}

\begin{proof}
Take an element
\begin{equation*}
\frac{M_1\cdots M_{j+1}}
{p_{I_1}\cdots p_{I_j}}
\in\mathcal B_{n,k}.
\end{equation*}
Put $B_i=I_i\setminus I_{i-1}$ and $d_i=\deg M_i$. By the definition of $\mathcal B_{n,k}$,
\begin{equation*}
0\leq d_i\leq |B_i|-1,\;\; d_i\equiv |B_i|-1\;(\textrm{mod}\; 2),\;\;\textrm{and}\;\; \sum_{i=1}^{j+1}d_i=2k+j-n+1.
\end{equation*}
Put
\begin{equation*}
h_i=\frac{|B_i|-1-d_i}{2}.
\end{equation*}
Then
\begin{equation*}
\sum_{i=1}^{j+1}h_i=\frac{n-j-1-(2k+j-n+1)}{2}=n-j-1-k.
\end{equation*}

By Lemma \ref{lem:block-realization}, for every $i$ there exists a multilinear polynomial $P_i$ on the variables indexed by $B_i$ such that $c_{P_i}=M_i$. Moreover, $P_i$ contains $h_i$ occurrences of $\succ$ and $|B_i|-1-h_i$ occurrences of $\prec$.

Put $T_1=P_1$, $T_{s+1}=T_s\succ P_{s+1}$ and $1\leq s\leq j$. The variables occurring in $T_s$ are indexed by $I_s$. Hence
\begin{equation*}
c_{T_{s+1}}
=
\frac{c_{T_s}c_{P_{s+1}}}{p_{I_s}},
\end{equation*}
and therefore
\begin{equation}\label{eq:left-chain-final}
c_{T_{j+1}}
=
\frac{M_1\cdots M_{j+1}}
{p_{I_1}\cdots p_{I_j}}.
\end{equation}
The total number of occurrences of $\succ$ in $T_{j+1}$ is $\sum_{i=1}^{j+1}h_i+j=n-1-k$, while the number of occurrences of $\prec$ is $\sum_{i=1}^{j+1}(|B_i|-1-h_i)=k$. Thus \eqref{eq:left-chain-final} belongs to $\mathcal R_{n,k}$.
\end{proof}

We next prove the spanning property. We first pass from a
$\prec,\succ$-monomial to a signed rooted tree.

\begin{lemma}\label{lem:signed-tree-form}
Let $T$ be a multilinear $\prec,\succ$-monomial in the variables indexed by a
finite set $B$. Then $T$ determines a rooted tree $\Gamma_T$ with vertex set
$B$ whose edges are labelled by $+$ and $-$. For an edge $e$, let $S_e$ be the
vertex set of the rooted subtree below $e$. Then
\begin{equation*}
c_T=\frac{\displaystyle\prod_{e:\,+}p_{S_e}}{\displaystyle\prod_{e:\,-}p_{S_e}}.
\end{equation*}
The positive edges correspond to the occurrences of $\prec$, and the negative
edges correspond to the occurrences of $\succ$.
\end{lemma}

\begin{proof}
The construction is recursive. A single variable $x_a$ gives the one-vertex
tree with no edges; both products are empty and $c_{x_a}=1$.

Suppose that $T=U\prec V$. Join the root of $\Gamma_V$ to the root of
$\Gamma_U$ by a positive edge $e$ and keep the root of $\Gamma_U$ as the root
of $\Gamma_T$. Then $S_e$ is the set of variables occurring in $V$, so the new
edge contributes the factor $p_V$.

If $T=U\succ V$, join the root of $\Gamma_U$ to the root of $\Gamma_V$ by a
negative edge $e$ and keep the root of $\Gamma_V$ as the root of $\Gamma_T$.
Then $S_e$ is the set of variables occurring in $U$, so the new edge
contributes the factor $p_U^{-1}$.

In both cases the new edge is attached at the root of one of the two trees.
Since no root lies below an already existing edge, the sets $S_e$ attached to
the edges of $\Gamma_U$ and of $\Gamma_V$ are unchanged in $\Gamma_T$.
Consequently the induction hypothesis applies to $U$ and $V$, and the formulas
\begin{equation*}
c_{U\prec V}=p_Vc_Uc_V,\qquad c_{U\succ V}=\frac{c_Uc_V}{p_U}
\end{equation*}
give the claim.
\end{proof}

For a vertex $C$ of a rooted tree whose vertices are labelled by pairwise
disjoint nonempty blocks, put
\begin{equation*}
\Bbbk[C]:=\Bbbk[\,p_a,\ p_{S_g}\mid a\in C,\ g\text{ a child edge of }C\,],
\end{equation*}
where $S_e$ denotes the union of the blocks lying below an edge $e$. The
displayed generators are linear forms with nonempty pairwise disjoint supports,
hence algebraically independent; every element of $\Bbbk[C]$ is therefore
uniquely a polynomial in them.

The next lemma describes the cancellation of denominator factors.

\begin{lemma}\label{lem:negative-edge-contraction}
Let $Q$ be such a tree and suppose that every vertex $C$ carries a homogeneous
$P_C\in\Bbbk[C]$ of degree $d_C$ with
\begin{equation*}
0\leq d_C\leq |C|-1,\qquad d_C\equiv |C|-1\pmod 2 .
\end{equation*}
Then
\begin{equation*}
\frac{\prod_C P_C}{\prod_{e\in E(Q)}p_{S_e}}
\end{equation*}
is a linear combination of fractions
\begin{equation*}
\frac{P_1\cdots P_{j+1}}{\prod_{e\in E(Q')}p_{S_e}},
\end{equation*}
where $Q'$ is obtained from $Q$ by contracting $c$ edges, $B_1,\ldots,B_{j+1}$ are the blocks of $Q'$, each $P_i$ is homogeneous of degree $d_i$ in the variables indexed by $B_i$ with
\begin{equation*}
0\leq d_i\leq |B_i|-1,\qquad d_i\equiv |B_i|-1\pmod 2 ,
\end{equation*}
and $\sum_{i=1}^{j+1}d_i=\sum_C d_C-c$.
\end{lemma}

\begin{proof}
Induction on the number of edges. If $Q$ has no edges, its only vertex has no child edges, so $P_C\in\Bbbk[\,p_a\mid a\in C\,]$ and the claim holds with $Q'=Q$, $j=0$, $c=0$.

Let $C$ be the root, let $e$ join a child block $C'$ to $C$, and put $X=p_{S_e}$. Deleting $e$ splits $Q$ into the subtree $Q_1$ with root $C'$ and the remaining tree $Q_2$ with root $C$. Since $e$ is incident to the root, no other edge has $C'$ below it; hence $S_{e'}$ is unchanged for every $e'\neq e$, and $S_e$ is the union of the blocks of $Q_1$. Contracting an edge merges two adjacent blocks and likewise leaves all other $S_{e'}$ unchanged. We use this below without further mention.

Write uniquely $P_C=P_C^{(0)}+XP_C^{(1)}$ with $P_C^{(0)}=P_C|_{X=0}$; by homogeneity $P_C^{(0)}$ and $P_C^{(1)}$ are, when nonzero, homogeneous of degrees $d_C$ and $d_C-1$.

For the term $P_C^{(0)}$ the denominator $X$ remains. As $P_C^{(0)}$ does not involve $X$, it lies in the ring allowed for the root of $Q_2$ and still has degree $d_C$, while the polynomials on $Q_1$ are unchanged. So the induction hypothesis applies to $Q_1$ and $Q_2$; multiplying the two linear combinations and restoring $X^{-1}$ gives terms in which $e$ is not contracted, with degree sum $\sum_C d_C-(c_1+c_2)$, where $c_i$ is the number of edges contracted in $Q_i$.

For the term $XP_C^{(1)}$ the factor $X$ cancels the denominator of $e$; we contract $e$ and put $C''=C\sqcup C'$. Since $P_C^{(1)}$ may still involve $X$, which is no longer the subtree sum of a child edge of $C''$, we substitute
\begin{equation*}
X=p_{C'}+\sum_g p_{S_g},
\end{equation*}
$g$ running through the child edges of $C'$; the right-hand side is a linear form in the generators of $\Bbbk[C'']$. Hence $P_{C''}=P_C^{(1)}P_{C'}$ lies in $\Bbbk[C'']$ and is homogeneous of degree $d_{C''}=d_C+d_{C'}-1$. If it is nonzero, then $d_C\geq1$, so $d_{C''}\geq0$; moreover
\begin{equation*}
d_{C''}\leq(|C|-1)+(|C'|-1)-1\leq|C''|-1,\qquad d_{C''}\equiv|C''|-1\pmod 2 .
\end{equation*}
Thus the contracted tree satisfies the hypotheses and has one edge fewer, while the degree sum drops by one and the number of contracted edges grows by one. The induction hypothesis now gives the assertion in this case as well.
\end{proof}

We shall also use the following elementary straightening of a rooted tree denominator.

\begin{lemma}\label{lem:tree-to-chain}
Let $Q$ be a rooted tree on pairwise disjoint nonempty blocks. Then
\begin{equation*}
\frac1{\displaystyle\prod_{e\in E(Q)}p_{S_e}}
\end{equation*}
is a linear combination of fractions
\begin{equation*}
\frac1{p_{I_1}\cdots p_{I_j}},
\end{equation*}
where $\varnothing\subsetneq I_1\subsetneq\cdots\subsetneq I_j\subsetneq[n]$ is obtained by ordering the vertex blocks according to a linear extension of the rooted-tree order, with descendants preceding their ancestors.
\end{lemma}

\begin{proof}
We use
\begin{equation}\label{eq:partial-fraction-basic}
\frac1{p_Ap_B}=\frac1{p_{A\cup B}}
\left(
\frac1{p_A}+\frac1{p_B}
\right), \qquad A\cap B=\varnothing.
\end{equation}
For a rooted path the statement is immediate. If the root has several branches, apply the induction hypothesis to every branch and then use \eqref{eq:partial-fraction-basic} repeatedly. This gives all shuffles of the chains corresponding to the branches, which are precisely the linear extensions of the rooted-tree order.
\end{proof}

It remains to reduce the polynomial numerator to the prescribed monomials.

\begin{lemma}\label{lem:chain-reduction}
Suppose that
\begin{equation*}
\varnothing=I_0\subsetneq I_1\subsetneq\cdots\subsetneq I_j\subsetneq I_{j+1}=[n],
\end{equation*}
and put $B_i=I_i\setminus I_{i-1}$. Let $P_i$ be homogeneous polynomials in the variables indexed by $B_i$ of degrees $d_i$ satisfying
\begin{equation*}
0\leq d_i\leq |B_i|-1,\qquad d_i\equiv |B_i|-1\pmod 2,
\end{equation*}
and $d_1+\cdots+d_{j+1}=2k+j-n+1$. Then
\begin{equation*}
\frac{P_1\cdots P_{j+1}}{p_{I_1}\cdots p_{I_j}}
\end{equation*}
is a linear combination of elements of $\mathcal B_{n,k}$.
\end{lemma}

\begin{proof}
We use induction on $j$. For $j=0$, the polynomial $P_1$ has degree $2k-n+1$ and satisfies the required degree conditions. Expanding it into monomials gives elements of $\mathcal B_{n,k}$.

Suppose that $j>0$. For $1\leq i\leq j$, put $b_i=\max B_i$. Let $\overline P_i$ be obtained from $P_i$ by the substitution
\begin{equation*}
p_{b_i}=-\sum_{a\in B_i\setminus\{b_i\}}p_a.
\end{equation*}
Then
\begin{equation*}
P_i=\overline P_i+p_{B_i}Q_i,
\end{equation*}
where $\overline P_i$ is homogeneous of degree $d_i$ in the variables
\begin{equation*}
\{p_a\mid a\in B_i\setminus\{\max B_i\}\},
\end{equation*}
and $Q_i$ is homogeneous of degree $d_i-1$ whenever it is nonzero. Since
\begin{equation*}
p_{B_i}=p_{I_i}-p_{I_{i-1}},\qquad p_{I_0}=0,
\end{equation*}
every term containing $p_{B_i}Q_i$ is a sum of fractions with one fewer denominator factor.

Suppose first that $p_{I_i}$ is cancelled. Then the adjacent blocks $B_i$ and $B_{i+1}$ merge, and the degree of the polynomial in the new block is $d_i+d_{i+1}-1$. If the corresponding term is nonzero, this degree is nonnegative. It also satisfies
\begin{equation*}
d_i+d_{i+1}-1\leq |B_i\cup B_{i+1}|-1\;\;\;\textrm{and}\;\;\;
d_i+d_{i+1}-1\equiv |B_i\cup B_{i+1}|-1\; (\textrm{mod} 2).
\end{equation*}
Moreover, the sum of all polynomial degrees decreases by one. Since the number of denominator factors also decreases from $j$ to $j-1$, we have
\begin{equation*}
\left(\sum_{s=1}^{j+1}d_s\right)-1=2k+(j-1)-n+1.
\end{equation*}
Thus the same conditions are preserved.

If $i>1$ and $p_{I_{i-1}}$ is cancelled, the blocks $B_{i-1}$ and $B_i$ merge, and the same argument applies.

Therefore the induction hypothesis applies to every term in which a denominator factor is cancelled. Repeating the decomposition for all nonfinal blocks, the only remaining term has each $P_i$, $1\leq i\leq j$, replaced by $\overline P_i$. Expanding these polynomials into monomials, and expanding $P_{j+1}$ into monomials in all variables of the final block, gives exactly the elements of $\mathcal B_{n,k}$.
\end{proof}

We can now prove the spanning property.

\begin{lemma}\label{lem:basis-spans}
For every $n$ and $k$,
\begin{equation*}
\mathcal R_{n,k}=\operatorname{span}\mathcal B_{n,k}.
\end{equation*}
\end{lemma}
\begin{proof}
The inclusion $\supseteq$ is Lemma~\ref{lem:basis-in-R}. Since $\mathcal R_{n,k}$ is spanned by the $c_T$ with $T$ a multilinear monomial containing $k$ occurrences of $\prec$, it suffices to show $c_T\in\operatorname{span}\mathcal B_{n,k}$ for every such $T$.

Let $\Gamma_T$ be the signed rooted tree of Lemma~\ref{lem:signed-tree-form}; it has $k$ positive and $n-1-k$ negative edges. Delete the negative edges and let $Q$ be the graph whose vertices are the resulting components and whose edges are the negative edges. Then $Q$ is obtained from the tree $\Gamma_T$ by contracting all positive edges, hence is a tree, rooted at the block containing the root of $\Gamma_T$, with $n-1-k$ edges; for a negative edge $e$ the blocks below $e$ in $Q$ have union $S_e$, so the notation is unambiguous. For a block $C$ put
\begin{equation*}
P_C=\prod_{\substack{a\ \text{positive edge}\\ a\subseteq C}}p_{S_a},\qquad\text{so that}\qquad c_T=\frac{\prod_C P_C}{\prod_{e\in E(Q)}p_{S_e}} .
\end{equation*}
The positive edges inside $C$ form a spanning tree of $C$, so $P_C$ is homogeneous of degree $d_C=|C|-1$; in particular the degree and parity conditions hold, and $\sum_C d_C=k$, the total number of positive edges. Note that $P_C$ is not a polynomial in the variables indexed by $C$ alone: for a positive edge $a\subseteq C$,
\begin{equation*}
p_{S_a}=\sum_{b\in S_a\cap C}p_b+\sum_g p_{S_g},
\end{equation*}
where $g$ runs through the child edges of $C$ lying below $a$, since $S_a\setminus C$ is the union of the subtrees hanging on these edges. Hence $P_C\in\Bbbk[C]$, as required in Lemma~\ref{lem:negative-edge-contraction}.

By that lemma, $c_T$ is a linear combination of fractions
\begin{equation*}
\frac{P_1\cdots P_{j+1}}{\prod_{e\in E(Q')}p_{S_e}},
\end{equation*}
where $Q'$ is obtained from $Q$ by contracting $c$ edges, its blocks $B_1,\ldots,B_{j+1}$ partition $[n]$, and each $P_i$ is homogeneous of degree $d_i$ in the variables indexed by $B_i$ with $0\le d_i\le|B_i|-1$ and $d_i\equiv|B_i|-1\pmod 2$. We treat each term separately, so $Q'$, $j$ and $c$ may vary from term to term. As $Q$ has $n-1-k$ edges and $Q'$ has $j$, we get $c=n-1-k-j$ and hence
\begin{equation*}
\sum_{i=1}^{j+1}d_i=k-(n-1-k-j)=2k+j-n+1 .
\end{equation*}
By Lemma~\ref{lem:tree-to-chain}, the denominator of such a term is a linear combination of chain denominators $p_{I_1}\cdots p_{I_j}$, where $I_s$ is the union of the first $s$ blocks of $Q'$ for a linear extension of the rooted-tree order. Setting $I_0=\emptyset$ and $I_{j+1}=[n]$, the sets $I_s\setminus I_{s-1}$ are exactly the blocks $B_1,\ldots,B_{j+1}$ in that order, with the polynomials $P_i$ attached to them; the displayed degree conditions are independent of the ordering. Lemma~\ref{lem:chain-reduction} now applies and expresses every such fraction as a linear combination of elements of $\mathcal B_{n,k}$.
\end{proof}

It remains to prove linear independence.

\begin{lemma}\label{lem:basis-independent}
The elements of $\mathcal B_{n,k}$ are pairwise distinct and linearly independent.
\end{lemma}

\begin{proof}
Suppose that there is a nontrivial linear relation among the functions \eqref{eq:basis-element-Rnk} associated with the data defining $\mathcal B_{n,k}$; this also covers the possibility that two different data give the same function. Let $j$ be the maximal number of denominator factors occurring in it and fix a chain
\begin{equation}\label{eq:independence-chain}
I_1\subsetneq\cdots\subsetneq I_j
\end{equation}
occurring in a term of maximal denominator length. Put $B_1=I_1$, $B_s=I_s\setminus I_{s-1}$ for $2\le s\le j$, and $B_{j+1}=[n]\setminus I_j$. 

Choose in each $B_s$, $1\le s\le j$, the variable $p_{\max B_s}$ and replace it by
\begin{equation*}
q_s=p_{B_s}=\sum_{a\in B_s}p_a .
\end{equation*}
Together with the variables $p_a$, $a\in B_i\setminus\{\max B_i\}$ for $1\le i\le j$ and $a\in B_{j+1}$, the elements $q_1,\ldots,q_j$ form a coordinate system on $\Bbbk^n$; all restrictions below are taken in these coordinates. In them, a subset $J\subseteq[n]$ satisfies
\begin{equation}\label{eq:vanishing-criterion}
p_J\equiv0 \text{ on } \{q_1=\cdots=q_s=0\}\iff J\text{ is a union of some of } B_1,\ldots,B_s,
\end{equation}
since the only linear relations imposed are $p_{B_1}=\cdots=p_{B_s}=0$. In particular the denominator of the fixed chain is $q_1(q_1+q_2)\cdots(q_1+\cdots+q_j)$.

We show by induction on $s$ that, after $s$ steps, only terms whose chain contains $I_1,\ldots,I_s$ survive. Multiply the relation by $q_1$ and restrict to $q_1=0$. By \eqref{eq:vanishing-criterion} with $s=1$, the only $p_J$ vanishing on this hyperplane is $p_{I_1}$, so every term is regular after multiplication by $q_1$, and every term whose chain does not contain $I_1$ becomes zero.

Assume the claim for $s-1$ and work modulo $q_1=\cdots=q_{s-1}=0$; the
surviving terms have had the factors $p_{I_1},\ldots,p_{I_{s-1}}$ cancelled, and
their remaining denominator factors do not vanish there. Let $p_J$ be such a
factor with $p_J\equiv0$ on $q_s=0$. By \eqref{eq:vanishing-criterion}, $J$ is a
union of blocks among $B_1,\ldots,B_s$, and $J\supseteq B_s$ because $J$ does
not already vanish. As $J$ belongs to a chain containing $I_{s-1}$, it is
comparable with $I_{s-1}$, and $J\not\subseteq I_{s-1}$; hence
$J\supseteq I_{s-1}$ and therefore $J=I_s$. So multiplying by $q_s$ and
restricting to $q_s=0$ is again defined on all surviving terms and kills those
whose chain avoids $I_s$.

After $j$ steps only terms whose chain contains all sets in
\eqref{eq:independence-chain} remain; since no chain in the relation is longer
than $j$, these are exactly the terms with chain \eqref{eq:independence-chain}.
For such a term the successive restrictions leave the numerator
$M_1\cdots M_{j+1}$, unchanged by the substitutions $q_1=\cdots=q_j=0$, because
$M_i$ does not involve $p_{\max B_i}$ for $i\le j$. By the definition of
$\mathcal B_{n,k}$ these products are monomials in the independent variables
$p_a$, $a\in B_i\setminus\{\max B_i\}$, $1\le i\le j$, and $p_a$,
$a\in B_{j+1}$; distinct data with the chain \eqref{eq:independence-chain}
give distinct monomials, since the blocks are disjoint. Hence all coefficients of the terms with this chain vanish.

Repeating this for every chain of length $j$ removes all terms of maximal denominator length, and descending on $j$ shows that all coefficients vanish. In particular distinct data give distinct functions, and these functions are linearly independent.
\end{proof}

We are now ready to prove the theorem.

\begin{proof}[Proof of Theorem \ref{thm:basis}] 
By Lemma \ref{lem:basis-spans}, $\mathcal B_{n,k}$ spans $\mathcal R_{n,k}$, and by Lemma \ref{lem:basis-independent} it is linearly independent. Hence $\mathcal B_{n,k}$ is a basis of $\mathcal R_{n,k}$. Finally, every element of $\mathcal R_{n,k}$ is homogeneous of degree $2k-n+1$. These degrees are different for different values of $k$. Therefore
\begin{equation*}
\mathcal R(n)=\bigoplus_{k=0}^{n-1}\mathcal R_{n,k},
\end{equation*}
and the union of the bases $\mathcal B_{n,k}$ over $0\leq k\leq n-1$ is a basis of $\mathcal R(n)$.
\end{proof}

\section{Dimension of the rational function model}

\begin{theorem}\label{thm:dimension-recurrence}
Put $d_n=\dim\mathcal R(n)$. For $s\geq1$, let
\begin{equation*}
a_s=
\sum_{\substack{0\leq d\leq s-1\\d\equiv s-1\pmod 2}}
\binom{s+d-2}{d}\;\;\;\textrm{and}\;\;\;
b_s=\sum_{\substack{0\leq d\leq s-1\\d\equiv s-1\pmod 2}}\binom{s+d-1}{d},
\end{equation*}
where $a_1=1$. Then
\begin{equation}\label{eq:dimension-recurrence}
d_n=b_n+\sum_{s=1}^{n-1}\binom ns a_s d_{n-s}, \qquad n\geq1.
\end{equation}
In particular,
\begin{equation*}
d_1=1,\;
d_2=4,\;
d_3=22,\;
d_4=152,\;
d_5=1291,\;
d_6=13156,\;
d_7=156682,\;
d_8=2133824.
\end{equation*}
\end{theorem}

\begin{proof}
By Theorem \ref{thm:basis},
\[
d_n=
\left|
\bigcup_{k=0}^{n-1}\mathcal B_{n,k}
\right|.
\]
Recall that a basis element is determined by a strict chain
\[
\varnothing=I_0\subsetneq I_1\subsetneq\cdots
\subsetneq I_j\subsetneq I_{j+1}=[n],
\]
with blocks $B_i=I_i\setminus I_{i-1}$, together with monomials $M_i$
satisfying
\[
0\leq \deg M_i\leq |B_i|-1,
\qquad
\deg M_i\equiv |B_i|-1\pmod 2.
\]
In the union over all $k$, these are the only degree restrictions, since
\[
k=
\frac{\sum_{i=1}^{j+1}\deg M_i-j+n-1}{2}
\]
is uniquely determined; the parity and degree conditions ensure that
$k$ is an integer with $0\leq k\leq n-1$.

If $j=0$, the only block is $B_1=[n]$. For each admissible degree $d$,
there are
\[
\binom{n+d-1}{d}
\]
monomials in $p_1,\ldots,p_n$, so this case contributes $b_n$.

Now let $j\geq1$ and put $s=|B_1|$. There are $\binom ns$ choices for
$B_1$. Since $B_1$ is nonfinal, $M_1$ is a monomial in the $s-1$
variables
\[
\{p_a\mid a\in B_1\setminus\{\max B_1\}\}.
\]
Thus the number of admissible choices for $M_1$ is $a_s$ (with
$a_1=1$ when $s=1$). The remaining blocks and monomials form an
arbitrary basis datum on the remaining $n-s$ variables, giving
$d_{n-s}$ possibilities.

Hence, for fixed $s$, the contribution is $\binom ns a_s d_{n-s}$. Summing over $1\leq s\leq n-1$ and adding the case $j=0$ yields \eqref{eq:dimension-recurrence}.
\end{proof}

\section{Free Novikov-Zinbiel algebra}

Let $\mathcal F$ be the free algebra with two binary operations $\prec$ and $\succ$, and let $J$ be the operadic ideal generated by the Zinbiel identity, the two Novikov identities, and the three compatibility identities from Definition~1. Put $\mathcal{NZ}=\mathcal F/J$. Thus $\mathcal{NZ}$ is the operad of Novikov-Zinbiel algebras. Since all defining identities hold for the operations
\begin{equation*}
x\succ y=R(x)y,\qquad x\prec y=xD(y),
\end{equation*}
the map $\rho_n$ factors through the multilinear component $\mathcal{NZ}(n)$. Hence there is a surjective linear map
\begin{equation*}
\overline\rho_n:\mathcal{NZ}(n)\longrightarrow\mathcal R(n).
\end{equation*}
The defining identities preserve the number of occurrences of $\prec$, and therefore
\begin{equation*}
\mathcal{NZ}(n)=\bigoplus_{k=0}^{n-1}\mathcal{NZ}_{n,k},
\end{equation*}
where $\mathcal{NZ}_{n,k}$ is spanned by multilinear monomials containing exactly $k$ occurrences of $\prec$. The map $\overline\rho_n$ restricts to a surjective map
\begin{equation*}
\overline\rho_{n,k}:\mathcal{NZ}_{n,k}\longrightarrow\mathcal R_{n,k}.
\end{equation*}
Our goal is to prove that these maps are injective.
    
Put
\begin{equation*}
E(x,y)=x\prec y+y\prec x
\end{equation*}
and define the ternary operation
\begin{equation*}
\mu(x,y,z)=E(x,y)\succ z.
\end{equation*}
The operation $\mu$ will describe the cancellations which occur inside the blocks.

\begin{lemma}\label{lem:ternary-product}
The ternary operation $\mu$ is symmetric and satisfies
\begin{equation}\label{eq:mu-right-prec}
\mu(x,y,z)\prec u
=
\mu(x\prec u,y,z),
\end{equation}
\begin{equation}\label{eq:mu-leibniz}
x\prec\mu(y,z,u)
=
\mu(x\prec y,z,u)
+
\mu(y,x\prec z,u)
+
\mu(y,z,x\prec u),
\end{equation}
and
\begin{equation}\label{eq:mu-left-succ}
x\succ\mu(y,z,u)
=
\mu(x\succ y,z,u).
\end{equation}
Moreover,
\begin{equation}\label{eq:mu-associativity}
\mu(\mu(x,y,z),u,v)
=
\mu(x,y,\mu(z,u,v)).
\end{equation}
\end{lemma}

\begin{proof}
The operation $\mu$ is symmetric in $x$ and $y$ by definition.
The third compatibility identity gives
\begin{equation*}
E(x,y)\succ z
=
E(x,z)\succ y,
\end{equation*}
and hence $\mu$ is symmetric in all three arguments.

To prove \eqref{eq:mu-right-prec}, using the symmetry of $\mu$ and the first compatibility identity, we obtain
\begin{multline*}
\mu(x\prec u,y,z)=\mu(y,z,x\prec u)=E(y,z)\succ(x\prec u)\\
=(E(y,z)\succ x)\prec u=\mu(y,z,x)\prec u=\mu(x,y,z)\prec u.
\end{multline*}

We next prove \eqref{eq:mu-leibniz}. The second compatibility identity gives
\begin{equation*}
x\prec\mu(y,z,u)
=(x\prec E(y,z))\succ u+(E(y,z)\prec x)\succ u+(E(y,z)\succ x)\prec u.
\end{equation*}
The last term is $\mu(y,z,x\prec u)$. It remains to show that
\begin{equation}\label{eq:E-Novikov-relation}
x\prec E(y,z)+E(y,z)\prec x
=
E(x\prec y,z)+E(y,x\prec z).
\end{equation}
The left-symmetric Novikov identity gives
\begin{equation*}
x\prec(y\prec z)+(y\prec x)\prec z
=
(x\prec y)\prec z+y\prec(x\prec z),
\end{equation*}
and, after interchanging $y$ and $z$,
\begin{equation*}
x\prec(z\prec y)+(z\prec x)\prec y
=
(x\prec z)\prec y+z\prec(x\prec y).
\end{equation*}
Adding these equalities and using right commutativity
\begin{equation*}
(a\prec b)\prec c=(a\prec c)\prec b
\end{equation*}
gives \eqref{eq:E-Novikov-relation}. Therefore
\eqref{eq:mu-leibniz} follows.

We now prove \eqref{eq:mu-left-succ}. By the first and second
compatibility identities,
\begin{multline*}
E(x\succ y,z)=(x\succ y)\prec z+z\prec(x\succ y)
=x\succ(y\prec z)+E(z,x)\succ y+(x\succ z)\prec y\\
=x\succ(y\prec z)+\mu(x,y,z)+x\succ(z\prec y).
\end{multline*}
Thus
\begin{equation*}
E(x\succ y,z)=x\succ E(y,z)+\mu(x,y,z).
\end{equation*}
Consequently,
\begin{equation*}
\mu(x\succ y,z,u)=E(x\succ y,z)\succ u=(x\succ E(y,z))\succ u+\mu(x,y,z)\succ u.
\end{equation*}
On the other hand, the Zinbiel identity gives
\begin{equation*}
x\succ\mu(y,z,u)=x\succ(E(y,z)\succ u)=(x\succ E(y,z))\succ u +(E(y,z)\succ x)\succ u.
\end{equation*}
Since $E(y,z)\succ x=\mu(y,z,x)=\mu(x,y,z)$, we obtain \eqref{eq:mu-left-succ}.

Finally, using \eqref{eq:mu-left-succ}, we have
\begin{equation*}
\mu(x,y,\mu(z,u,v))=E(x,y)\succ\mu(z,u,v)=\mu(E(x,y)\succ z,u,v)=\mu(\mu(x,y,z),u,v).
\end{equation*}
which proves \eqref{eq:mu-associativity}.
\end{proof}

By symmetry, \eqref{eq:mu-associativity} allows an inner occurrence of $\mu$ to be moved between any two positions. Hence all iterated products of an odd number of arguments are independent of the bracketing. We shall therefore write
\begin{equation*}
\Pi(a_1,\ldots,a_{2h+1}):=\mu\bigl(\Pi(a_1,\ldots,a_{2h-1}),a_{2h},a_{2h+1}\bigr).
\end{equation*}
for their iterated product with respect to $\mu$. We also use the convention $\Pi(a)=a$.

The next lemma gives two rules which will be used repeatedly.

\begin{lemma}\label{lem:Pi-rules}
For arbitrary $a_1,\ldots,a_{2h+1},u$ and any $i$, we have
\begin{equation}\label{eq:Pi-transport}
\Pi(a_1,\ldots,a_{2h+1})\prec u=\Pi(a_1,\ldots,a_i\prec u,\ldots,a_{2h+1}).
\end{equation}
and
\begin{equation}\label{eq:Pi-leibniz}
x\prec\Pi(a_1,\ldots,a_{2h+1})=\sum_{i=1}^{2h+1}\Pi(a_1,\ldots,x\prec a_i,\ldots,a_{2h+1}).
\end{equation}
\end{lemma}

\begin{proof}
For $h=0$, both identities are immediate. For $h=1$, \eqref{eq:Pi-transport} is \eqref{eq:mu-right-prec}, together with the symmetry of $\mu$, while \eqref{eq:Pi-leibniz} is \eqref{eq:mu-leibniz}. Suppose now that $h>1$. Write
\begin{equation*}\Pi(a_1,\ldots,a_{2h+1})=\mu(\Pi(a_1,\ldots,a_{2h-1}),a_{2h},a_{2h+1}).
\end{equation*}
Formula \eqref{eq:Pi-transport} follows from \eqref{eq:mu-right-prec}, the symmetry of $\mu$, and induction. For \eqref{eq:Pi-leibniz}, apply \eqref{eq:mu-leibniz} to the last display. The first term is
\begin{equation*}
\mu(x\prec\Pi(a_1,\ldots,a_{2h-1}),a_{2h},a_{2h+1}),
\end{equation*}
and the induction hypothesis expands it into the first $2h-1$ terms of \eqref{eq:Pi-leibniz}. The remaining two terms are exactly the terms corresponding to $a_{2h}$ and $a_{2h+1}$.
\end{proof}

In particular, \eqref{eq:Pi-transport} gives the transport identity
\begin{equation}\label{eq:transport}
\Pi(a\prec u,b,a_3,\ldots,a_{2h+1})=\Pi(a,b\prec u,a_3,\ldots,a_{2h+1}).
\end{equation}

We shall also use a property of the pure Novikov part. Let $\mathcal N(B)$ denote the multilinear component of the free Novikov algebra on the variables indexed by $B$. By the standard differential realization of the free Novikov algebra \cite{DzhLofwall2002}, the coefficient map identifies $\mathcal N(B)$ with the space of homogeneous polynomials of degree $|B|-1$ in the variables $p_a$, $a\in B$. Thus every such monomial has a unique preimage in $\mathcal N(B)$.

Fix $c\in B$. A pure $\prec$-monomial is called $c$-anchored if $x_c$ lies in the left argument of every operation $\prec$ on the path from $x_c$ to the root. Equivalently, $x_c$ lies on the leftmost branch of the corresponding binary tree.

\begin{lemma}\label{lem:anchored-novikov}
Let $P\in\mathcal N(B)$. If its coefficient does not depend on $p_c$,
then $P$ is a linear combination of $c$-anchored monomials.
\end{lemma}

\begin{proof}
Let $A_c$ be the span of all $c$-anchored monomials. Let $K_c$ be the subspace of $\mathcal N(B)$ consisting of the elements whose coefficients are divisible by $p_c$. We first show that
\begin{equation}\label{eq:anchored-decomposition}
\mathcal N(B)=A_c+K_c.
\end{equation}
It is enough to prove this for a pure $\prec$-monomial $T$. We use lexicographic induction. The first parameter is the number of variables of $T$. If
\begin{equation*}
T=U\prec V
\end{equation*}
and $x_c$ belongs to $V$, the second parameter is the number of variables in $V$; if $x_c$ belongs to $U$, we take the second parameter to be zero.

Suppose first that $x_c$ occurs in $U$. Applying induction to $U$, we may write
\begin{equation*}
U=A+K,\qquad A\in A_c,
\end{equation*}
where the coefficient of $K$ is divisible by $p_c$. Then
\begin{equation*}
T=A\prec V+K\prec V.
\end{equation*}
The first term belongs to $A_c$, while the coefficient of the second term is divisible by $p_c$.

Suppose now that $x_c$ occurs in $V$. Applying induction to $V$, and working term by term, we may assume that either the coefficient of $V$ is divisible by $p_c$, in which case the same is true for $U\prec V$, or $V$ is $c$-anchored.

If $V=x_c$, then
\begin{equation*}
c_{U\prec x_c}=p_c\,c_U,
\end{equation*}
so $U\prec x_c\in K_c$. Otherwise,
\begin{equation*}
V=V_1\prec V_2,
\end{equation*}
where $V_1$ is $c$-anchored. The left-symmetric Novikov identity gives
\begin{equation*}
U\prec(V_1\prec V_2)=(U\prec V_1)\prec V_2-(V_1\prec U)\prec V_2+V_1\prec(U\prec V_2).
\end{equation*}
Using right commutativity in the first term, we obtain
\begin{equation*}
(U\prec V_1)\prec V_2=(U\prec V_2)\prec V_1.
\end{equation*}
The last two terms are $c$-anchored. In the first term, $x_c$ again
belongs to the right factor, but this factor is $V_1$, which contains
strictly fewer variables than $V$. Hence the second induction
parameter decreases. This proves \eqref{eq:anchored-decomposition}.

Now write
\begin{equation*}
P=A+K,
\qquad
A\in A_c,
\quad
K\in K_c.
\end{equation*}
The coefficient of every $c$-anchored monomial is independent of
$p_c$, and therefore $c_A$ is independent of $p_c$. By assumption,
$c_P$ is also independent of $p_c$. Hence
\begin{equation*}
c_K=c_P-c_A
\end{equation*}
is independent of $p_c$. At the same time, it is divisible by $p_c$.
Therefore
\begin{equation*}
c_K=0.
\end{equation*}
The differential realization of the free Novikov algebra is
injective, so $K=0$. Thus $P=A$.
\end{proof}

We now describe the normal form inside a single block.

A \emph{block expression} is a multilinear expression obtained from
variables using only $\prec$ and the ternary operation $\mu$. Let
$\mathcal C_{B,h}$ be the span of block expressions on $B$ containing
exactly $h$ occurrences of $\mu$.

If $|B|=r$, then such an expression contains
\begin{equation*}
d=r-1-2h
\end{equation*}
binary occurrences of $\prec$ outside the occurrences of $\mu$.
Consequently, its coefficient is a homogeneous polynomial of degree
$d$.

\begin{lemma}\label{lem:block-normal-form}
Let $B$ be a set of $r$ variables and let
\begin{equation*}
0\leq d\leq r-1,
\qquad
d\equiv r-1\pmod 2.
\end{equation*}
Put
\begin{equation*}
h=\frac{r-1-d}{2}.
\end{equation*}
Then the coefficient map identifies $\mathcal C_{B,h}$ with the
space of homogeneous polynomials of degree $d$ in the variables
$p_a$, $a\in B$.

More precisely, for every monomial $M$ of degree $d$ there is a
well-defined element
\begin{equation*}
\beta_B(M)\in\mathcal C_{B,h}
\end{equation*}
such that
\begin{equation}\label{eq:block-beta-coefficient}
c_{\beta_B(M)}=M,
\end{equation}
and the elements $\beta_B(M)$ form a basis of $\mathcal C_{B,h}$.
\end{lemma}

\begin{proof}
We first show that every block expression with $h$ occurrences of $\mu$ is a linear combination of expressions
\begin{equation}\label{eq:Pi-pure-form}
\Pi(P_1,\ldots,P_{2h+1}),
\end{equation}
where all $P_i$ are pure $\prec$-polynomials on pairwise disjoint sets of variables.

We use structural induction. The statement is immediate for a
variable. If
\begin{equation*}
T=\mu(T_1,T_2,T_3),
\end{equation*}
apply induction to the three arguments and use the total
associativity of $\mu$ to combine the resulting iterated products
into one expression of the form \eqref{eq:Pi-pure-form}.

It remains to consider
\begin{equation*}
T=T_1\prec T_2.
\end{equation*}
By induction, it is enough to consider
\begin{equation*}
T_1=\Pi(P_1,\ldots,P_{2r+1})\;\;\textrm{and}\;\;T_2=\Pi(Q_1,\ldots,Q_{2s+1}),
\end{equation*}
where all $P_i$ and $Q_j$ are pure $\prec$-polynomials. By
\eqref{eq:Pi-leibniz},
\begin{equation*}
T_1\prec T_2
=
\sum_{j=1}^{2s+1}
\Pi(
Q_1,\ldots,T_1\prec Q_j,\ldots,Q_{2s+1}
).
\end{equation*}
By \eqref{eq:Pi-transport},
\begin{equation*}
T_1\prec Q_j
=
\Pi(P_1\prec Q_j,P_2,\ldots,P_{2r+1}).
\end{equation*}
Using total associativity once more, each resulting nested iterated
product becomes an expression of the form
\eqref{eq:Pi-pure-form}. This proves the claim.

We next move all nontrivial pure Novikov structure into the first
argument of \eqref{eq:Pi-pure-form}. If, for some $i>1$,
\begin{equation*}
P_i=U\prec V,
\end{equation*}
then \eqref{eq:transport} gives
\begin{equation*}
\Pi(P_1,\ldots,U\prec V,\ldots)
=
\Pi(P_1\prec V,\ldots,U,\ldots).
\end{equation*}
The number of variables contained in non-singleton arguments outside
the first position strictly decreases. Repeating this procedure, we
obtain
\begin{equation}\label{eq:block-core-form}
\Pi(P,x_{a_1},\ldots,x_{a_{2h}}),
\end{equation}
where $P$ is a pure $\prec$-polynomial on
\begin{equation*}
r-2h=d+1
\end{equation*}
variables.

Let $C\subseteq B$ be the set of variables occurring in $P$. The
differential realization of the free Novikov algebra identifies the
multilinear component on $C$ with the space of homogeneous
polynomials of degree $d$ in the variables $p_a$, $a\in C$. Hence
$P$ is a linear combination of the unique elements corresponding to
monomials of degree $d$.

Let now $M$ be a monomial of degree $d$ in the variables indexed by
$B$. Choose a subset
\begin{equation*}
C\subseteq B,
\qquad
|C|=d+1,
\end{equation*}
containing the support of $M$. Let $P_C(M)$ denote the unique element
of the multilinear free Novikov algebra on $C$ whose coefficient is
$M$. We claim that
\begin{equation}\label{eq:beta-core-definition}
\Pi(P_C(M),x_{a_1},\ldots,x_{a_{2h}}),
\qquad
B\setminus C=\{a_1,\ldots,a_{2h}\},
\end{equation}
does not depend on the choice of $C$.

It is enough to consider two cores which differ by one variable. Let
\begin{equation*}
c\in C\setminus\operatorname{supp}M
\end{equation*}
and
\begin{equation*}
z\in B\setminus C.
\end{equation*}
Since the coefficient of $P_C(M)$ does not depend on $p_c$, Lemma
\ref{lem:anchored-novikov} allows us to choose a $c$-anchored
representative $P$ of $P_C(M)$.

For every such representative,
\begin{equation}\label{eq:core-exchange}
\Pi(P,x_z,\ldots)
=
\Pi(P^{c\mapsto z},x_c,\ldots),
\end{equation}
where $P^{c\mapsto z}$ is obtained by replacing $x_c$ by $x_z$.

We prove \eqref{eq:core-exchange} by induction on the number of
variables of $P$. For $P=x_c$, it follows from the symmetry of
$\Pi$. Otherwise,
\begin{equation*}
P=U\prec V,
\end{equation*}
where $U$ is $c$-anchored. By induction,
\begin{equation*}
\Pi(U,x_z,\ldots)
=
\Pi(U^{c\mapsto z},x_c,\ldots).
\end{equation*}
Applying $\prec V$ to both sides and using
\eqref{eq:Pi-transport}, we obtain
\begin{equation*}
\Pi(U\prec V,x_z,\ldots)
=
\Pi(U^{c\mapsto z}\prec V,x_c,\ldots),
\end{equation*}
which is \eqref{eq:core-exchange}.

The coefficient of $P^{c\mapsto z}$ is again $M$. Hence, by
uniqueness in the free Novikov differential realization,
\begin{equation*}
P^{c\mapsto z}=P_{(C\setminus\{c\})\cup\{z\}}(M).
\end{equation*}
Thus the expression \eqref{eq:beta-core-definition} is unchanged by
replacing one zero-exponent variable in the core by one outside the
core.

Any two subsets of size $d+1$ containing the support of $M$ can be
connected by a sequence of such exchanges. At each step Lemma
\ref{lem:anchored-novikov} can be applied to the zero-exponent
variable which is to be removed. Therefore
\eqref{eq:beta-core-definition} depends only on $B$ and $M$. We denote
this element by
\begin{equation*}
\beta_B(M).
\end{equation*}

By Lemma 2, or directly from the definition of $\mu$, the coefficient
of an iterated product $\Pi$ is the product of the coefficients of
its arguments. Hence
\begin{equation*}
c_{\beta_B(M)}=M,
\end{equation*}
which proves \eqref{eq:block-beta-coefficient}.

The first part of the proof shows that the elements $\beta_B(M)$
span $\mathcal C_{B,h}$. If
\begin{equation*}
\sum_M\lambda_M\beta_B(M)=0,
\end{equation*}
then applying the rational-function model gives
\begin{equation*}
\sum_M\lambda_M M=0.
\end{equation*}
Distinct polynomial monomials are linearly independent. Therefore
all $\lambda_M$ are zero, and the elements $\beta_B(M)$ form a basis.
\end{proof}

For every homogeneous polynomial $F=\sum_M\lambda_M M$ of an admissible degree, we shall use the notation
\begin{equation}\label{eq:beta-linear-extension}
\beta_B(F)=\sum_M\lambda_M\beta_B(M).
\end{equation}
By Lemma \ref{lem:block-normal-form}, $c_{\beta_B(F)}=F$. We now pass from blocks to chains. For block expressions $P_1,\ldots,P_s$, put
\begin{equation*}
[P_1|\cdots|P_s]=(((P_1\succ P_2)\succ P_3)\cdots)\succ P_s.
\end{equation*}

\begin{lemma}\label{lem:block-chains-span}
The multilinear component $\mathcal{NZ}(n)$ is spanned by block chains.
\end{lemma}

\begin{proof}
Let $\mathcal C$ be the span of all block chains. Every generator belongs to $\mathcal C$. We prove that $\mathcal C$ is closed under both operations.

First consider $\succ$. Let $A$ and $B$ be block chains. We use induction on the total number of blocks in $A$ and $B$. If $B$ consists of one block, then $A\succ B$ is already a block chain. Otherwise write
\begin{equation*}
B=B'\succ P,
\end{equation*}
where $P$ is the last block. The Zinbiel identity gives
\begin{equation*}
A\succ(B'\succ P)=(A\succ B')\succ P+(B'\succ A)\succ P.
\end{equation*}
Both products inside the parentheses involve fewer blocks, so the induction hypothesis applies. Hence $\mathcal C$ is closed under $\succ$.

We now consider $A\prec B$. If $A$ has more than one block, write
\begin{equation*}
A=A'\succ P.
\end{equation*}
The first compatibility identity gives
\begin{equation*}
A\prec B=A'\succ(P\prec B).
\end{equation*}
Thus it remains to consider the case where the left argument is a single block $P$.

If the right argument is also a single block $Q$, then
\begin{equation*}
P\prec Q
\end{equation*}
is again a block expression. Otherwise write
\begin{equation*}
B=B'\succ Q,
\end{equation*}
where $Q$ is the last block. The second compatibility identity, together with the first one, gives
\begin{equation}\label{eq:prec-chain-reduction}
P\prec(B'\succ Q)=\mu(P,B',Q)+B'\succ(P\prec Q).
\end{equation}
The second term is a block chain. If $B'$ consists of one block, then the first term is a block expression. Otherwise write
\begin{equation*}
B'=B''\succ R,
\end{equation*}
where $R$ is its last block. By symmetry of $\mu$ and
\eqref{eq:mu-left-succ},
\begin{equation*}
\mu(P,B''\succ R,Q)=B''\succ\mu(R,P,Q),
\end{equation*}
which is again a block chain.

Thus $\mathcal C$ is closed under both operations. Since it contains all generators, it coincides with the whole free
Novikov-Zinbiel algebra.
\end{proof}

By Lemma \ref{lem:block-normal-form}, every block in a chain may now be written in the form $\beta_B(F)$ for an admissible homogeneous polynomial $F$. We next reduce the polynomial of every nonfinal block to the canonical variables used in the definition of $\mathcal B_{n,k}$.

\begin{lemma}\label{lem:abstract-chain-reduction}
Every multilinear block chain is a linear combination of chains
\begin{equation}\label{eq:canonical-NZ-chain}
[\beta_{B_1}(M_1)|\cdots|\beta_{B_{j+1}}(M_{j+1})],
\end{equation}
where $[n]=B_1\sqcup\cdots\sqcup B_{j+1}$, and, for every $1\leq i\leq j$, the monomial $M_i$ involves only the variables
\begin{equation*}
\{p_a\mid a\in B_i\setminus\{\max B_i\}\}.
\end{equation*}
The final monomial $M_{j+1}$ may involve all variables of $B_{j+1}$.
\end{lemma}

\begin{proof}
We use induction on the number $j$ of separators $\succ$ between the blocks. For $j=0$ there is nothing to prove. Consider a nonfinal block $B_i$. Put $b_i=\max B_i$ and write its block expression as
\begin{equation*}
\beta_{B_i}(F_i),
\end{equation*}
where $F_i$ is homogeneous of an admissible degree $d$. Let $\overline F_i$ be obtained from $F_i$ by the substitution
\begin{equation*}
p_{b_i}=-\sum_{a\in B_i\setminus\{b_i\}}p_a.
\end{equation*}
Then
\begin{equation}\label{eq:block-polynomial-reduction}
F_i=\overline F_i+p_{B_i}Q_i,
\end{equation}
where $\overline F_i$ is homogeneous of degree $d$ and does not involve $p_{b_i}$, while $Q_i$ is homogeneous of degree $d-1$ whenever it is nonzero.

By the linearity of \eqref{eq:beta-linear-extension}, it remains to consider a monomial term $p_{B_i}q$ from the second summand in \eqref{eq:block-polynomial-reduction}. Let $r=|B_i|$. Then
\begin{equation*}
\deg q=d-1\leq r-2.
\end{equation*}
Hence some variable $p_a$, $a\in B_i$, does not occur in $q$. Put
\begin{equation*}
U=x_a\;\;\textrm{and}\;\;V=\beta_{B_i\setminus\{a\}}(q).
\end{equation*}
The degree $d-1$ is admissible for $B_i\setminus\{a\}$, because
\begin{equation*}
0\leq d-1\leq r-2,\qquad d-1\equiv r-2\pmod2.
\end{equation*}

Let
\begin{equation*}
h=\frac{r-1-d}{2}.
\end{equation*}
The element $V$ contains exactly $h$ occurrences of $\mu$, since
\begin{equation*}
\frac{(r-1)-1-(d-1)}{2}=h.
\end{equation*}
Thus
\begin{equation*}
E(U,V)=U\prec V+V\prec U
\end{equation*}
belongs to the same block space $\mathcal C_{B_i,h}$ as $\beta_{B_i}(p_{B_i}q)$. Its coefficient is
\begin{equation*}
c_{E(U,V)}=p_{B_i\setminus\{a\}}q+p_aq=p_{B_i}q.
\end{equation*}
By the injectivity in Lemma \ref{lem:block-normal-form},
\begin{equation}\label{eq:pB-block}
\beta_{B_i}(p_{B_i}q)=E(U,V).
\end{equation}

We now show that a block of the form $E(U,V)$ can be removed as a separate block. Suppose first that it is not the first block and let $A$ denote the chain preceding it, while $W$ is the following block. The Zinbiel identity gives
\begin{equation}\label{eq:separator-cancellation}
(A\succ E(U,V))\succ W=A\succ\mu(U,V,W)-\mu(A,U,V)\succ W.
\end{equation}
Indeed,
\begin{equation*}
A\succ\mu(U,V,W)=A\succ(E(U,V)\succ W)=(A\succ E(U,V))\succ W+(E(U,V)\succ A)\succ W,
\end{equation*}
and, by symmetry of $\mu$,
\begin{equation*}
E(U,V)\succ A=\mu(U,V,A)=\mu(A,U,V).
\end{equation*}

The first term on the right-hand side of \eqref{eq:separator-cancellation} merges the block $B_i$ with the following block.

For the second term, if $A$ consists of one block, then
\begin{equation*}
\mu(A,U,V)
\end{equation*}
is a single block, so $B_i$ is merged with the preceding block. If $A$ has more than one block, write
\begin{equation*}
A=A'\succ P.
\end{equation*}
By \eqref{eq:mu-left-succ},
\begin{equation*}
\mu(A,U,V)=A'\succ\mu(P,U,V),
\end{equation*}
and again the last block of $A$ is merged with $B_i$. Thus both terms in \eqref{eq:separator-cancellation} have one fewer block separator.

If $B_i$ is the first block, then simply
\begin{equation*}
E(U,V)\succ W=\mu(U,V,W),
\end{equation*}
so one separator disappears in this case as well.

If further blocks occur after $W$, the same identities are applied to the corresponding prefix of the left-associated chain, and the remaining factors are appended afterwards. Therefore every term coming from $p_{B_i}Q_i$ is a linear combination of block chains with one fewer separator. The induction hypothesis applies to all such terms.

We repeat the decomposition \eqref{eq:block-polynomial-reduction} for every nonfinal block. The only terms which retain $j$ separators are those in which every $F_i$, $1\leq i\leq j$, has been replaced by $\overline F_i$. Expanding $\overline F_i$ into monomials in
\begin{equation*}
\{p_a\mid a\in B_i\setminus\{\max B_i\}\}
\end{equation*}
and expanding the final block polynomial into monomials in all its variables gives exactly the chains \eqref{eq:canonical-NZ-chain}.
\end{proof}

We can now prove the main result.

\begin{theorem}\label{thm:NZ-rational-isomorphism}
For every $n\geq1$ and $0\leq k\leq n-1$, the induced map
\begin{equation*}
\overline\rho_{n,k}:\mathcal{NZ}_{n,k}\longrightarrow\mathcal R_{n,k}
\end{equation*}
is an isomorphism. More precisely, for every
\begin{equation*}
b=\frac{M_1\cdots M_{j+1}}{p_{I_1}\cdots p_{I_j}}\in\mathcal B_{n,k},
\end{equation*}
where $B_i=I_i\setminus I_{i-1}$, the element
\begin{equation}\label{eq:canonical-lift}
\widetilde b=[\beta_{B_1}(M_1)|\cdots|\beta_{B_{j+1}}(M_{j+1})]
\end{equation}
belongs to $\mathcal{NZ}_{n,k}$, and the elements $\widetilde b$, $b\in\mathcal B_{n,k}$, form a basis of $\mathcal{NZ}_{n,k}$.
\end{theorem}

\begin{proof}
We first verify the grading of \eqref{eq:canonical-lift}. Put $d_i=\deg M_i$ and
\begin{equation*}
h_i=\frac{|B_i|-1-d_i}{2}.
\end{equation*}
By the construction in Lemma \ref{lem:block-normal-form}, $\beta_{B_i}(M_i)$ contains exactly $h_i$ occurrences of $\mu$. After expanding every occurrence of $\mu$, each of them contributes one occurrence of $\prec$ and one occurrence of $\succ$. The pure Novikov core contains $d_i$ occurrences of $\prec$. Hence every monomial occurring in $\beta_{B_i}(M_i)$ contains $d_i+h_i$ occurrences of $\prec$. Therefore the number of occurrences of $\prec$ in \eqref{eq:canonical-lift} is
\begin{align*}
\sum_{i=1}^{j+1}(d_i+h_i)=\frac{\sum_{i=1}^{j+1}d_i+n-j-1}{2}=\frac{2k+j-n+1+n-j-1}{2}=k.
\end{align*}
Thus $\widetilde b\in\mathcal{NZ}_{n,k}$.

By Lemma \ref{lem:block-chains-span}, every multilinear element of $\mathcal{NZ}(n)$ is a linear combination of block chains. By Lemma \ref{lem:abstract-chain-reduction}, every such chain is a linear combination of canonical chains of the form \eqref{eq:canonical-lift}. All defining identities preserve the number of occurrences of $\prec$, so an element of $\mathcal{NZ}_{n,k}$ is reduced only to canonical chains containing exactly $k$ occurrences of $\prec$.

For a canonical chain let $d_i=\deg M_i$. The preceding count gives
\begin{equation*}
k=\frac{\sum_{i=1}^{j+1}d_i+n-j-1}{2}\;\;\;\textrm{or}\;\;\; \sum_{i=1}^{j+1}d_i=2k+j-n+1.
\end{equation*}
Thus its chain, blocks, and monomials satisfy exactly the defining conditions of $\mathcal B_{n,k}$. Hence the elements \eqref{eq:canonical-lift}, with $b\in\mathcal B_{n,k}$, span $\mathcal{NZ}_{n,k}$.

It remains to prove their linear independence. By Lemma
\ref{lem:block-normal-form},
\begin{equation*}
c_{\beta_{B_i}(M_i)}=M_i.
\end{equation*}
The successive blocks in \eqref{eq:canonical-lift} are joined by $\succ$. Since $B_1\sqcup\cdots\sqcup B_s=I_s$, we obtain
\begin{equation*}
\overline\rho_{n,k}(\widetilde b)=\frac{M_1\cdots M_{j+1}}{p_{I_1}\cdots p_{I_j}}=b.
\end{equation*}

Suppose that $\sum_{b\in\mathcal B_{n,k}}\lambda_b\widetilde b=0$. Applying $\overline\rho_{n,k}$ gives $\sum_{b\in\mathcal B_{n,k}}\lambda_b b=0$. By Theorem \ref{thm:basis}, the elements of $\mathcal B_{n,k}$ are linearly independent. Therefore $\lambda_b=0$ for every $b\in\mathcal B_{n,k}$.

Thus the elements $\widetilde b$, $b\in\mathcal B_{n,k}$, form a basis of $\mathcal{NZ}_{n,k}$, and $\overline\rho_{n,k}$ maps this basis bijectively onto $\mathcal B_{n,k}$. Hence $\overline\rho_{n,k}$ is an isomorphism.
\end{proof}

Taking the direct sum over $k$ gives the following consequence.

\begin{corollary}\label{cor:NZ-rational-isomorphism}
For every $n\geq1$, the map
\begin{equation*}
\overline\rho_n:\mathcal{NZ}(n)\longrightarrow\mathcal R(n)
\end{equation*}
is an isomorphism. In particular,
\begin{equation*}
\dim\mathcal{NZ}(n)=\dim\mathcal R(n).
\end{equation*}
\end{corollary}

Finally, we obtain the completeness of the defining identities.

\begin{corollary}\label{cor:complete-identities1}
 Let $A$ be a commutative-associative algebra with an invertible derivation $D$, and put $R=D^{-1}$. Define 
 \begin{equation*}
x\succ y=R(x)y,\qquad x\prec y=xD(y).
\end{equation*}
Then $(A,\succ,\prec)$ is a free Novikov-Zinbiel algebra. 
 \end{corollary}

\begin{proof}
Let $I(n)$ be the space of all multilinear identities satisfied by the construction. Since the six defining identities hold in every such algebra, $J(n)\subseteq I(n)$. On the other hand, the algebra $t\Bbbk[t]$ belongs to this class, and therefore $I(n)\subseteq\ker\rho_n$. By Theorem~3, $\ker\rho_n=J(n)$. Hence $I(n)=J(n)$.
\end{proof}

\subsection*{Acknowledgments}
This research was funded by the Science Committee of the Ministry of Education and Science of the Republic of Kazakhstan (Grant No. AP32517747).

\end{document}